\documentclass[12pt]{amsart}

\usepackage[margin=3cm]{geometry} 
\usepackage[T1]{fontenc}
\usepackage[utf8]{inputenc}
\usepackage[english]{babel}
\usepackage{amssymb, amsthm, enumerate, enumitem, graphicx, amscd, amsmath, amssymb, amsfonts, mathrsfs, xcolor, tikz, tikz-cd, tkz-euclide}

\usepackage{orcidlink}

\usepackage{enumitem}
\usepackage{hyperref}
\usepackage{xcolor}
\usepackage{orcidlink}
\usepackage{comment}
\usepackage{tikz}

\usepackage{centernot}
\usepackage{mathtools}

\usepackage[mathcal]{euscript}
\usepackage{hyperref}
\hypersetup{
colorlinks=true,
linktoc=all,
linkcolor=black,
citecolor=darkgray
}

\title[Rhaly operators and factorable matrices]{Rhaly operators on weighted Hardy spaces \\ and factorable matrices}

\numberwithin{equation}{section}
\newtheorem{theorem}{Theorem}
\numberwithin{theorem}{section}
\newtheorem{definition}[theorem]{Definition}

\newtheorem{lemma}[theorem]{Lemma}

\newtheorem{proposition}[theorem]{Proposition}

\def\CB{\color{black} }

\def\CBL{\color{blue} }

\newtheorem{thmA}{Theorem}

\author[C. Bellavita]{Carlo Bellavita}
\email{carlo.bellavita@gmail.com}
\email{carlo.bellavita@unimi.it}
\address{Dipartimento di Matematica ``F. Enriques''\\
  Dipartimento di Eccellenza MUR 2023-2027\\
Universit\'a degli Studi di Milano\\
Via C. Saldini 50\\
20133 Milano, Italy.}

\author[E. Dellepiane]{Eugenio Dellepiane}
\email{dellepianeeugenio@gmail.com}
\address{D\'epartement de math\'ematiques et de statistique, Universit\'e Laval, Qu\'ebec, QC, Canada G1V 0A4.}

\author[G. Stylogiannis]{Georgios Stylogiannis}
\email{g.stylog@gmail.com}
\email{stylog@math.auth.gr}
\address{Department of Mathematics, Aristotle University of Thessaloniki, 54124, Greece.}


\begin{document}

\maketitle

\begin{abstract}
    In this paper, we study properties of Rhaly operators acting on weighted Hardy spaces $H^2(\omega)$. This problem is intimately related to the study of factorable matrices on $\ell^2$. In particular, our main contributions are in the study of Schatten class properties of such operators.
\end{abstract}

\section{Introduction}
Given a sequence of positive real numbers $\omega=(\omega_n)_n$, the associated weighted Hardy space $H^2(\omega)$ is defined as 
\begin{equation*}
    H^2(\omega) := \left\{f(z) = \sum_{n=0}^\infty a_n z^n \in \operatorname{Hol}(\mathbb{D}) : \|f\|^2_\omega = \sum_{n=0}^\infty |a_n|^2 \omega_n < \infty \right\}.
\end{equation*}

{Here, $\operatorname{Hol}(\mathbb{D})$ denotes the space of analytic functions on the unit disk $\mathbb{D}$.} This family of analytic function spaces includes several important examples. Setting $\omega_n \equiv 1$, one recovers the classical Hardy space $H^2$. Taking $\omega_n = n^{1-\alpha}$ for $n\geq 1$ yields the weighted Dirichlet spaces $\mathcal{D}_\alpha$. Finally, if $\omega_n = (n+1)^{-\gamma-1}$, one obtains the weighted Bergman spaces $A^2_\gamma$. More general examples of Bergman spaces with radially symmetric weights are also included, see for example \cite{BMNP, DS, HKZ}.

Given a sequence of complex numbers $\eta=(\eta_n)_n$, we define the Rhaly operator $R_\eta$, which acts on sequences $x=(x_n)_n$ as 
\begin{equation*}
    (R_\eta x)_n = \eta_n \sum_{j=0}^n x_j, \qquad n\in\mathbb{N}.
\end{equation*}
By identifying the analytic function $f(z) = \sum_n a_n z^n$ with the sequence of its Taylor coefficients, we can consider the action of $R_\eta$ on $\operatorname{Hol}(\mathbb{D})$ via the formula
\[
R_\eta f(z) = \sum_{n=0}^\infty \left(\eta_n \sum_{j=0}^n a_j\right) z^n, \qquad z\in\mathbb{D}.
\]

In this article, we are interested in characterizing the boundedness, the compactness and the Schatten class membership of the Rhaly matrices acting on weighted Hardy spaces in terms of the sequences $\eta$ and $\omega$. Our first main theorem deals with the boundedness.

\begin{theorem}\label{T:boundcompFactorableweighted} Consider a sequence $\eta$ and a weight $\omega$. The following are equivalent.
\begin{itemize}
    \item The Rhaly operator $R_\eta$ is bounded on $H^2(\omega)$.
    \item One has that
  \[
   K_1:=\sup_{m\in\mathbb{N}} \left(\sum_{j=0}^m \frac{1}{\omega_j}\right)^{-1}\sum_{k=0}^m |\eta_k|^2\omega_k\left(\sum_{j=0}^k \frac{1}{\omega_j}\right)^2<\infty.
  \]
    \item One has that
    \[
 K_2:=\sup_{m\in\mathbb{N}} \left(\sum_{k=m}^\infty |\eta_k|^2\omega_k\right)\left(\sum_{j=0}^m \frac{1}{\omega_j}\right)<\infty.
  \]
\end{itemize}
Moreover, $K_1\leq 2K_2^2\leq {2}\|R_\eta\|_{\mathcal{B}(H^2(\omega))}^2\leq  8 K_1. $
\end{theorem}
Analogous characterizations can be given for compactness, using little-$o$ conditions.
When $\omega\equiv 1$, one reobtains the results on the classical Hardy space $H^2$. We will show that specializing to the cases of weighted Bergman and Dirichlet spaces $A^2_\gamma$ and $\mathcal{D}_\alpha$ one recovers results of Galanopoulos--Girela \cite{GG2025} and Blasco--Galanopoulos--Girela \cite{Blasco2026Cesarotype}.

{The key ingredient of our proofs is the following: the action of} $R_\eta$ on $H^2(\omega)$ is modeled by a \emph{factorable} matrix acting on $\ell^2$. A factorable matrix $F_{\alpha,\beta}$ is
an infinite lower triangular matrix whose non-zero entries $(a_{k,j})_{j \leq k}$ can be factored as the product of two sequences $\alpha=(\alpha_k)_k$ and $\beta=(\beta_j)_j$. More explicitly, $F_{\alpha,\beta}$ is given by
\begin{equation*}
F_{\alpha,\beta}=\begin{pmatrix}
\alpha_0 \beta_0 & 0 & 0  & \cdots\\
\alpha_1 \beta_0 &  \alpha_1 \beta_1& 0  & \cdots\\
 \alpha_2\beta_0 &  \alpha_2\beta_1& \alpha_2\beta_2  & \cdots \\
\vdots & \vdots & \vdots  & \ddots
\end{pmatrix}.
\end{equation*}

In the literature, there are several complete characterizations for the boundedness and compactness of factorable matrices acting from $\ell^p$ to $\ell^q$ for $0<p,q<\infty$ see \cite{bennett, BENNETT1991, GrosseErdmann1998}. {The following theorem is a collection of known results (see \cite[Theorem 2]{bennett} and \cite[Theorem 9.4]{GrosseErdmann1998}).}
\begin{thmA} \label{T:factorablebound}
Let $F_{\alpha,\beta}=(\alpha_k\beta_j)_{j \leq k}$ be a factorable matrix. Consider the sequence $\rho$ defined by
\begin{equation}\label{def of rho}
\rho_m:=\left(\sum_{k=m}^\infty |\alpha_k|^2\right)^{\frac{1}{2}}\left(\sum_{j=0}^m |\beta_j|^2\right)^{\frac{1}{2}}.    \end{equation}
\begin{itemize}
    \item {$F_{\alpha,\beta}\colon \ell^2\to\ell^2$ is bounded if and only if $\rho$ is a bounded sequence.}
    \item {$F_{\alpha,\beta}\colon \ell^2\to\ell^2$ is compact if and only if $\lim_{m \to \infty} \rho_m=0$.}
\end{itemize}
{Moreover, the following operator norm inequalities hold}
$$\|\rho\|_{\infty} \leq \| F_{\alpha,\beta}\|_{\mathcal{B}(\ell^2)}  \leq 2\sqrt{2} \|\rho\|_{\infty}.$$
\end{thmA}


We move on to the discussion of the Schatten class properties of factorable matrices. To {the authors'} knowledge, a first complete account on the matter was given in \cite{BDSjlms2025}, in the case of \emph{Rhaly matrices}, {that are factorable matrices with $\beta\equiv 1$ or, equivalently, Rhaly operators on the classical Hardy space $H^2$.}

Given the statements of Theorem \ref{T:factorablebound}, it feels natural that a similar result {involving the sequence $\rho$} should hold for the $p$-Schatten class of $F_{\alpha,\beta}$. However, we show that the situation is more complicated than it may initially appear. In general, a direct equivalence stating that "$F_{\alpha,\beta}$ belongs to the $p$-Schatten class $\mathcal{S}^p(\ell^2)$ if and only if the sequence $\rho$ belongs to $\ell^p$" is false, even in the case of Rhaly matrices.

\begin{theorem} \label{T:IntroSchattenrho}
{Consider a parameter $1<p<\infty$ and two sequences $\alpha\in\ell^2$ and $\beta\notin\ell^2$. If the sequence $\rho$ defined in $(\ref{def of rho})$ belongs to $\ell^p$, then the factorable matrix $F_{\alpha,\beta}$ belongs to the $p$-Schatten class $\mathcal{S}^p(\ell^2)$. However, the converse is false: there exist sequences $\alpha\in\ell^2$ and $\beta\notin\ell^2$ such that $F_{\alpha,\beta}\in\mathcal{S}^p(\ell^2)$ but $\rho \notin \ell^p$. The sequence $\beta$ can be chosen to be the constant $1$.}
\end{theorem}

This theorem is in fact a corollary of a much more precise result. In Theorem \ref{T:thmSchattenrho}, we prove that $F_{\alpha,\beta}\in\mathcal{S}^p(\ell^2)$ is equivalent to a particular subsequence $(\rho_{\xi_k})_{k\in\mathbb{N}}$ belonging to $\ell^p$. In the case of Rhaly matrices (i.e., $\beta\equiv 1$), it turns out that $\xi_k=2^k$. The condition that is given in \cite{BDSjlms2025} to characterize the $p$-Schatten class of Rhaly operators is a dyadic condition. Even if it {may be} more unpleasant to handle, Theorem \ref{T:thmSchattenrho} suggests that a dyadic condition is natural in this context.

As the reader will see, the analysis that led to Theorem \ref{T:thmSchattenrho} is rather abstract. However, with some assumptions on $\beta$, we produce explicit summability conditions that characterize when a factorable matrix belongs to a certain Schatten class. 

\begin{theorem}\label{T:betapolynomial}
    Let $\alpha\in\ell^2$ and $1<p<\infty$. Consider also a sequence $\beta$ for which there exist an exponent $\delta>-\frac{1}{2}$ and constants $C_1,C_2>0$ such that
    \[
    C_1 (n+1)^\delta \leq |\beta_n| \leq C_2(n+1)^\delta, \qquad n\geq 0.
    \]
    Then, the factorable matrix $F_{\alpha,\beta}$ belongs to the Schatten class $\mathcal{S}^p(\ell^2)$ if and only if 
    \[
    \sum_{k=0}^\infty \left(2^{(2\delta+1)k}\sum_{j=2^k}^{2^{k+1}-1} |\alpha_j|^2\right)^{\frac{p}{2}}<\infty.
    \]
\end{theorem}

Notice that for $\delta=0$ one obtains also the Rhaly matrices. In this case, one recovers the statement of Theorem 1.3 of \cite{BDSjlms2025}. 

{An important consequence of Theorem \ref{T:betapolynomial} is the following: we are able to characterize the Schatten class membership of Rhaly operators acting on weighted Bergman and Dirichlet spaces.}

\begin{theorem} \label{T:SchattenBergman}
    Consider a sequence $\eta$ and a parameter $1<p<\infty$. The following are equivalent.
    \begin{itemize}
        \item For $\gamma>-1$, the Rhaly operator $R_\eta$ is in the Schatten class of the weighted Bergman space $\mathcal{S}^p(A^2_\gamma)$.
        \item For $0<\alpha\leq 1$, the Rhaly operator $R_\eta$ is in the Schatten class of the weighted Dirichlet space $\mathcal{S}^p(\mathcal{D}_\alpha)$.
        \item The sequence $\eta$ satisfies the summability condition
     \begin{equation} \label{E:SchattenBergman}
         \sum_{k=0}^\infty \left(2^k\sum_{j=2^k}^{2^{k+1}-1} |\eta_j|^2\right)^{\frac{p}{2}}<\infty. 
     \end{equation}
    \end{itemize} 
\end{theorem}

Notice that the condition \eqref{E:SchattenBergman} does not depend on the parameters $\gamma$ or $\alpha$, and it is the same one that characterizes the $p$-Schatten class of the Rhaly operator $R_\eta$ acting on the classical Hardy space $H^2$. For the classical Dirichlet space $\mathcal{D}_0$, we will show that this independency breaks down.

The paper is organized as follows.
In Section \ref{S:prelim} we collect the necessary background and technical preliminaries. 
Section \ref{S:Schatten} is devoted to the Schatten class theory of
factorable matrices. 
In Section \ref{S:weightedHardy} we turn to weighted Hardy spaces.
We compute the matrix representation of the Rhaly operator $R_\eta$ on
$H^2(\omega)$ with respect to the natural orthonormal basis
$\{\varphi_n\}$, and show that {$R_\eta$ is} unitarily equivalent to {a factorable
matrix}.
This identification reduces Theorems \ref{T:boundcompFactorableweighted}
and \ref{T:SchattenBergman} to the abstract results of
Section \ref{S:Schatten}, and allows us to recover the boundedness
characterizations of Galanopoulos--Girela \cite{GG2025} and
Blasco--Galanopoulos--Girela \cite{Blasco2026Cesarotype} for the weighted Bergman and Dirichlet spaces. 



\section{Preliminaries}\label{S:prelim}

We briefly recall some standard definitions and properties. Given a Hilbert space $H$, we denote by $\mathcal{B}(H)$ the space of bounded linear operators on $H$ and by $\|\cdot\|_{\mathcal{B}(H)}$ the operator norm, defined as
\[
\|T\|_{\mathcal{B}(H)} := \sup_{\|f\|_H = 1} \|Tf\|_H.
\]
Let $\mathcal{F}_n(H)$ be the set of all bounded linear operators on $H$ whose range has dimension less than or equal to $n$, and let $\mathcal{F}(H)$ be the set of all bounded finite-range operators on $H$. Given $T \in \mathcal{B}(H)$, its approximation numbers are defined as
\[
a_{n+1}(T) := \inf_{P \in \mathcal{F}_n(H)} \|T-P\|_{\mathcal{B}(H)}, \qquad n \in \mathbb{N}.
\]    
Note that $a_1(T) = \|T\|_{\mathcal{B}(H)}$. For $1 < q < \infty$, we say that the operator $T$ belongs to the Schatten class $\mathcal{S}^q(H)$ if and only if the sequence of its approximation numbers $(a_n(T))_{n\in\mathbb{N}}$ belongs to $\ell^q$. For further details, we refer the reader to \cite[Chapter 1]{zhu2007operator}.

We clarify the notation for the rest of the paper. We will be dealing with sequences with two indices. When we say that a sequence $(\alpha_{k,m})_{k,m\in\mathbb{N}}$ belongs to $\ell^p_k$, we are saying that it is $p$-summable with respect to the index $k$. Moreover, we write
\[
\|\alpha_{k,m}\|_{\ell^p_k}^p:=\sum_{k\in\mathbb{N}}|\alpha_{k,m}|^p, \qquad m\in\mathbb{N}.
\]
For $p>0$, we also denote the standard $\ell^p$ (quasi-)norm by $\|\cdot\|_p$. Given two functions $f,g$ defined on $\mathbb{N}$, by $f=O(g)$ we mean that there exists a positive constant $C$ such that $|f(n)|\leq C|g(n)|$, either for every $n\in\mathbb{N}$ or for $n$ big enough (it will be clear from context). If $f=O(g)$ and $g=O(f)$, we write that $f\asymp g$.
Let $e_n = (\delta_{n,k})_{k\in\mathbb{N}}$ denote the $n$-th vector of the canonical orthonormal basis of $\ell^2$.

\subsection{The sequence $\xi$}
We adapt techniques originally developed by Edmunds, Evans, and Harris in \cite{Edmunds1988,Edmunds1997} for studying integral operators on $L^p(\mathbb{R})$ to our discrete setting. The same approach was used in \cite{BDSjlms2025} to study Rhaly operators on the Hardy space $H^2$, which corresponds to the specific case where $\beta \equiv 1$. Adapting these techniques to the more general case of factorable matrices (where the sequence $\beta$ is not identically one) presents several technical obstacles. As we will explain, an additional assumption on the sequence $\beta$ will be required {in some instances}.

In what follows, we exclusively consider sequences $\beta \notin \ell^2$. We introduce the non-decreasing function
\[
B(n) := \sum_{j=0}^n |\beta_j|^2, \qquad n \in \mathbb{N},
\]
and the corresponding sequence  
\begin{equation*}
    \xi_k := \inf\{n \in \mathbb{N} \colon B(n) > 2^k\}.
\end{equation*}
The natural number $\xi_k$ is always well-defined because $\lim_{n \to \infty} B(n) = +\infty$ for $\beta \notin \ell^2$. {It is possible that $\xi_k = \xi_{k+1}$. We may even have an arbitrarily large number of consecutive $\xi_k$ values that coincide, though not infinitely many.  For every $k$, the definition of infimum yields
\(
B(\xi_k-1) \leq 2^k < B(\xi_k).
\)
}

{
Considering then another sequence $\alpha\in\ell^2$, we  define the sequence $\sigma = (\sigma_k)_{k\in\mathbb{N}}$ as 
\begin{equation}\label{D:defsigma}
\sigma_k := \left(\sum_{j=\xi_k}^{\xi_{k+1}-1} B(j)|\alpha_j|^2\right)^{\frac{1}{2}}.  
\end{equation}
Whenever $\xi_k = \xi_{k+1}$, we set $\sigma_k = 0$ by convention. The assumption that $\alpha\in\ell^2$ is not restrictive. Assuming without loss of generality that $\beta_0\neq 0$, the identity $F_{\alpha,\beta}e_0=\beta_0\alpha$ holds true. In particular, $\alpha\in\ell^2$ is a necessary condition in order to have $F_{\alpha,\beta}\in\mathcal{B}(\ell^2)$.
}

By the inequalities \(
B(\xi_k-1) \leq 2^k < B(\xi_k)
\) and the monotonicity of $B$, we obtain the bounds
\begin{equation}\label{E:sigmabounds}
        2^k \sum_{j=\xi_k}^{\xi_{k+1}-1}|\alpha_j|^2 \leq \sigma_k^2 \leq  2^{k+1} \sum_{j=\xi_k}^{\xi_{k+1}-1}|\alpha_j|^2.
    \end{equation}

When $\beta \equiv 1$, we recover the exact quantities considered in \cite{BDSjlms2025}: $B(n) = n+1$ for $n \in \mathbb{N}$, which implies
\[
\xi_k = \inf\{n \in \mathbb{N} \colon n+1 > 2^k\} = 2^k, \qquad k \in \mathbb{N},
\]
and
\[
\sigma_k = \left(\sum_{j=2^k}^{2^{k+1}-1} (j+1)|\alpha_j|^2\right)^{\frac{1}{2}}, \qquad k \in \mathbb{N}.  
\]

We relate the boundedness and compactness of $F_{\alpha,\beta}$ to the sequence $\sigma = (\sigma_k)_k$. 

\begin{proposition}\label{P:boundRasigma}
The factorable matrix $F_{\alpha,\beta}$ is bounded on $\ell^2$ if and only if the sequence $\sigma$ is bounded. In particular, we have the estimates
\[
\|\sigma\|_\infty \leq \sqrt{2}\|F_{\alpha,\beta}\|_{\mathcal{B}(\ell^2)} \leq 8\|\sigma\|_{\infty}.
\] 
Moreover, $F_{\alpha,\beta}$ is compact on $\ell^2$ if and only if $\lim_{k \to \infty} \sigma_k = 0$.
\end{proposition}
  
\begin{proof}
Let $m \geq 1$, and let $k$ be the unique natural number such that $m \in [\xi_k, \xi_{k+1}-1]$. If $\xi_l = \xi_{l+1} = \dots = \xi_{l+m}$, we choose $\xi_k = \xi_{l+m}$. Then,
\begin{align*}
   \left(\sum_{j=m}^\infty |\alpha_j|^2\right)\left(\sum_{i=0}^m |\beta_i|^{2}\right) &\leq 2^{k+1}\sum_{j=\xi_k}^\infty |\alpha_j|^2 = 2^{k+1}\sum_{h=k}^\infty\sum_{j=\xi_h}^{\xi_{h+1}-1} |\alpha_j|^2\\
   &\leq  2^{k+1}\sum_{h=k}^\infty 2^{-h} \sigma_h^2 \leq 4\sup_{h\geq k} \sigma_h^2.
\end{align*}
By Theorem \ref{T:factorablebound}, if $\sigma$ is bounded, then $F_{\alpha,\beta}$ is bounded, and 
$\|F_{\alpha,\beta}\|_{\mathcal{B}(\ell^2)} \leq 4\sqrt{2}\|\sigma\|_{\infty}$. Furthermore, if $\lim_{k \to \infty} \sigma_k = 0$, then $F_{\alpha,\beta}$ is compact. 

On the other hand, notice that $\sigma_k \leq 2 \rho_{\xi_k}$, since
 \[
 \sigma_{k}^{2} \leq 2^{k+1}\sum_{j=\xi_k}^{\infty}|\alpha_j|^{2} = 2\frac{2^{k}B(\xi_k)}{B(\xi_k)}\sum_{j=\xi_k}^{\infty}|\alpha_j|^{2} \leq 2\left(\sum_{i=0}^{\xi_k} |\beta_i|^2\right)\left(\sum_{j=\xi_k}^{\infty}|\alpha_j|^{2}\right).
 \]
Again by Theorem \ref{T:factorablebound}, if $F_{\alpha,\beta}$ is bounded, we have $\|\sigma\|_\infty \leq \sqrt{2} \|F_{\alpha,\beta}\|_{\mathcal{B}(\ell^2)}$, and if $F_{\alpha,\beta}$ is compact, then $\lim_{k \to \infty} \sigma_k = 0$.
\end{proof}

\subsection{The $(\epsilon,L)$-sequences}
We fix two sequences $\alpha \in \ell^2$ and $\beta \notin \ell^2$. In what follows, we consider natural numbers $a \leq b$ and intervals $I = [a,b]$. We introduce the notation
\[
\|f\|_{2,I} := \left(\sum_{k \in I}|f(k)|^2 \right)^{\frac{1}{2}}.
\]
Given an interval $I = [a,b]$ and $f \in \ell^2$, we set 
\begin{align*}
    l(I,f) &:= \sum_{k\in I}\sum_{n\in I\setminus\{k\}}\left|\alpha_k \alpha_n\sum^{\max(k,n)}_{j=\min(k,n)+1}f(j)\beta_j\right|^{2}
\end{align*}
and
\begin{equation}\label{L(I) def} 
    L(I) := \left(\sup_{\|f\|_{2,I} \leq 1} \frac{l(I,f)}{\sum_{k\in I}|\alpha_k|^2} \right)^{\frac{1}{2}}.
\end{equation}
If $\alpha_k = 0$ for every $k \in I$, we set $L(I) = 0$. According to this definition, $L(\{k\}) = 0$ for every singleton $\{k\}$. This is reminiscent of the continuous setting of \cite{Edmunds1988,Edmunds1997}, in which these techniques were first used. 

The quantity $L([a, b])$ decreases as $a$ increases and it increases as $b$ increases: the proof to this statement is analogous to that of Lemma 3.2 in \cite{BDSjlms2025}, the presence of $\beta_j$ does not change anything. To handle infinite intervals, we set
\[
L([a, \infty)) := \lim_{b \to \infty} L([a, b]) = \sup_{b \geq a} L([a,b]).
\] 
Clearly, the finiteness of $L([0,\infty))$ depends only on the tail behavior: $L([0,\infty)) < \infty$ if and only if there exists an $a \in \mathbb{N}$ such that $L([a,\infty)) < \infty$. This quantity is intimately related to the boundedness of $F_{\alpha,\beta}$.

\begin{lemma}\label{L:Fuv well def}
    If $L([0,\infty))<\infty,$ then $F_{\alpha,\beta}f\in \ell^2$ for every $f\in\ell^2$, and $F_{\alpha,\beta}\colon\ell^2\to\ell^2$ is bounded.
\end{lemma}
\begin{proof}
    Without loss of generality, we can assume that $\alpha_0\neq 0$. Take $f$ in $\ell^2$ with norm $1$. For every $b>0$, setting $I:=[0,b]$, we have that
    \begin{align*}
        \bigg(\sum_{k=0}^\infty |\alpha_k|^2\bigg) L([0,+\infty))^2 &\geq l([0,b],f) \\
        &=\sum_{k=0}^b\sum_{n\in I\setminus\{k\}}\biggl|\alpha_k \alpha_n\sum^{\max(k,n)}_{j=\min(k,n)+1}f(j)\beta_j\biggr|^{2}\\
        &\geq \sum_{n=1}^b\biggl|\alpha_0 \alpha_n\sum^{n}_{j=1}f(j)\beta_j\biggr|^{2}\\
        &=|\alpha_0|^2 \sum_{n=1}^b \biggl|F_{\alpha,\beta}f(n)-\alpha_nf(0)\beta_0\biggr|^2.
    \end{align*}
    It follows that
    \begin{align*}
        \sum_{n=1}^b |F_{\alpha,\beta}f(n)|^2 &\leq \frac{2}{|\alpha_0|^{2}}\bigg(\sum_{k=0}^\infty |\alpha_k|^2\bigg) L([0,+\infty))^2 +2|f(0)\beta_0|^{2}\sum_{n=1}^b |\alpha_n|^2\\
        &\leq 2\|\alpha\|_2^2\left( |\alpha_0|^{-2} L([0,+\infty))^2+|f(0)\beta_0|^{2}\right)<\infty,
    \end{align*}
    for every $b>0$. Then, $F_{\alpha,\beta}f\in\ell^2$ and  $F_{\alpha,\beta}$ is bounded on $\ell^2$  by the closed graph theorem.
\end{proof}

We are now ready to introduce the $(\epsilon,L)$-sequences, a key tool developed in \cite{Edmunds1988,Edmunds1997} and adapted to this discrete setting in \cite{BDSjlms2025}.

\begin{definition} 
Given a real number $\epsilon > 0$, we define a sequence of natural numbers $c_k$ recursively: we set $c_0 = 0$, and for $k \geq 0$,
\begin{equation}\label{def of c_k}
     c_{k+1} = \inf\{t \in \mathbb{N} \colon t > c_{k}, \quad L([c_{k},t-1]) > \epsilon\}.
\end{equation}
We say that $\{c_k\}$ forms an $(\epsilon, L)$-sequence. 
\end{definition}

If there exists a natural number $N$ such that $c_N \in \mathbb{N}$ and $c_{N+1} = +\infty$, we say that the $(\epsilon, L)$-sequence $\{c_k\}_{k=0}^N$ is finite and has length $N$. Otherwise, the sequence is infinite.

The elements $c_k$ of an $(\epsilon,L)$-sequence have the following crucial properties:
\begin{itemize}
    \item $c_{k+1} > c_{k} + 1$ for every $k \in \mathbb{N}$;
    \item $L([c_k,c_{k+1}-2]) \leq \epsilon$;
    \item If $c_{k+1} \in \mathbb{N}$, then $L([c_k,c_{k+1}-1]) > \epsilon$.
\end{itemize}

We can now relate the properties of $F_{\alpha,\beta}$ to the length of the $(\epsilon,L)$-sequences. 

\begin{theorem}\label{T:boundRaeLseq} 
The following conditions are equivalent:
\begin{enumerate}[label=(\roman*)]
    \item The operator $F_{\alpha,\beta}$ is bounded on $\ell^2$.
    \item $L([0, \infty)) < \infty$.
    \item There exists an $(\epsilon,L)$-sequence of finite length.
    \item The sequence $\sigma$ is bounded.
\end{enumerate} 
\end{theorem}

Before discussing the proof, we also state a parallel result characterizing the compactness of $F_{\alpha,\beta}$.

\begin{theorem} \label{T:compRaeLseq2}
    The following conditions are equivalent:
    \begin{enumerate}[label=(\roman*)]
        \item $F_{\alpha,\beta}$ is a compact operator on $\ell^2$.
        \item $\mathcal{L} := \lim_{n \to \infty} L([n,\infty)) = 0$.
        \item Every $(\epsilon,L)$-sequence has finite length.
        \item $\lim_{k \to \infty} \sigma_k = 0$.
    \end{enumerate}
\end{theorem}

These two theorems are generalizations of Theorem 4.4 and Theorem 4.6 in \cite{BDSjlms2025}. The novelty is that now the sequence $\beta$ appears. 
We briefly comment on the proof of Theorem \ref{T:boundRaeLseq}, to show how to take care of $\beta$. {The equivalence $(i)\iff(iv)$ was established in Proposition \ref{P:boundRasigma}. The equivalence $(ii)\iff(iii)$ follows by the definitions. A proof for the implication \mbox{$(ii)\implies(i)$} is contained in Lemma \ref{L:Fuv well def}. To prove that $(i)\implies(ii)$, a deeper analysis of the quantity $L(I)$ is needed. Given an interval $I=[a,b]$, we introduce the following auxiliary quantities.} First, $J(I):=\inf_{c \in I} \lbrace  \max (A_{c}^I,B_{c}^I)\rbrace$, where
\begin{align*}\label{A(c) def}
  A_{c}^I&:=\sup_{a\leq s\leq c}\bigg(\sum_{j=s+1}^c|\beta_j|^2\bigg)^{\frac{1}{2}}\biggl(\sum_{k=a}^{s}|\alpha_k|^{2}\biggr)^{\frac{1}{2}}, \\
  B_{c}^I&:=\sup_{c\leq s\leq b} \bigg(\sum_{j=c}^{s-1}|\beta_j|^2\bigg)^{\frac{1}{2}}\biggl(\sum_{k=s}^{b}|\alpha_k|^{2}\biggr)^{\frac{1}{2}}.
\end{align*} 
Finally, we consider \begin{equation*}
K(I) := \sup_{\|f\|_{2,I}\leq 1} \biggl(\sum_{k\in I} |S_f(k)-(S_f)_I|^2 |\alpha_k|^2\biggr)^{\frac{1}{2}},
\end{equation*}
where \[
f_{I}:=\frac{1}{\sum_{k\in I}|\alpha_k|^2}\sum_{k\in I}f(k)|\alpha_k|^{2}, \qquad S_f(k):=\sum_{j=0}^{k}f(j)\beta_j, \quad k\in\mathbb{N}.
\]
Notice that $F_{\alpha,\beta} f(k)=\alpha_k S_f(k).$  Just as in Lemma 3.6 in \cite{BDSjlms2025}, the quantities $L(I), K(I), J(I)$ satisfy the following relation:
\begin{equation}\label{E:K and L equi}
    \frac{1}{2}J(I)\leq K(I) = \frac{1}{\sqrt{2}}L(I)\leq 2J(I).
\end{equation}
The proof is analogous to that of Lemma 3.6 in \cite{BDSjlms2025}. We only point out the following slight adjustments. To obtain the lower bound $J(I)\leq 2K(I)$, one can choose as test functions in the definition of $K$ respectively $\chi_{[K+1,N]}\overline{\beta}$ and $\chi_{[N,M-1]}\overline{\beta}$. To prove the upper bound for $K$, instead of considering the auxiliary function $g(j)=(c+1-j)^{1/4}$, one should take \(g(j)=(\sum_{l=j}^c|\beta_l|^2)^{1/4}.\) In \cite{BDSjlms2025}, $g$ was estimated via a simple integral. Now, one should use Lemma 1' of \cite{Bennett2004}.

{
Once Equation \eqref{E:K and L equi} is settled, one can use it to prove the following results on the approximation numbers of the factorable matrix $F_{\alpha,\beta}$, that are analogous to Lemmas 4.2 and 4.3 in \cite{BDSjlms2025}:
\begin{itemize}
    \item[$(a)$] Assume that $F_{\alpha,\beta}$ is bounded. Take $\epsilon > 0$, and suppose that the associated $(\epsilon,L)$-sequence has finite length $N$. Then, 
\(a_{2N+2}(R_\alpha)\leq \frac{\epsilon}{\sqrt{2}}.\)
\item[$(b)$] Assume that $F_{\alpha,\beta}$ is bounded. Take $\epsilon > 0$ and suppose that the $(\epsilon,L)$-sequence has length at least $N\in\mathbb{N}$. Then, \( a_{N}(R_{\alpha})\geq\frac{\epsilon}{\sqrt{2}}.\)
\end{itemize}
}
{We point out that a nice feature of the technique of the $(\epsilon,L)$-sequences is that it does not simply provides a Schatten class membership criteria but it gives estimates for the singular values. These estimates are also important for the analysis of other properties of the operator. Indeed, as in the proof of Theorem 4.4 in \cite{BDSjlms2025}, one can use the result in item $(b)$ to prove that $L([0,+\infty))<\sqrt{2}\|F_{\alpha,\beta}\|_{\mathcal{B}(\ell^2)}$, thus proving the implication $(i)\implies(ii)$ and completing the proof of Theorem \ref{T:boundRaeLseq}. The proof of Theorem \ref{T:compRaeLseq2} is analogous of that of Theorem 4.5 in \cite{BDSjlms2025}.}


\section{Schatten class of factorable matrices}\label{S:Schatten}
It was already discussed in \cite{BDSjlms2025} that the adaptation of the $(\epsilon,L)$-sequences techniques from the continuous setting to the discrete one has some critical issues. Adding the sequence $\beta$ creates more complications. In order to have a complete characterization for the Schatten class membership of factorable matrices, we will have to restrict to a special class of sequences. We will elaborate on this throughout this section, but first we start off with a result that does not require any assumption.

\begin{theorem}\label{T:sigmaimpliesSchatten} Let $\alpha\in\ell^2$ and $\beta$ not in $\ell^2$. For $p\in (1,+\infty)$, if $\sigma\in\ell^p$, then $F_{\alpha,\beta}$ belongs to the $p$-Schatten class $\mathcal{S}^p(\ell^2)$.
\end{theorem}

\begin{proof}
The proof is analogous to that of Theorem 5.7 of \cite{BDSjlms2025}, \emph{mutatis mutandis}. 
\end{proof}

\subsection{Necessary condition} 
For the reverse statement of Theorem \ref{T:sigmaimpliesSchatten}, we will need an extra assumption on the sequence $\beta$. In the continuous case treated by Edmunds--Evans--Harris, $\xi_k$ is chosen as the real number such that $B(\xi_k)=2^k$ (here, $B$ is an integral, not a sum). In our discrete setting, it is always true that $B(\xi_k)\geq 2^k$, as $\xi_k$ is chosen exactly as the first point $k$ for which $B(k)> 2^k$. However, in general, we do not have a uniform control from above of the form $B(\xi_k)\leq C 2^k$.

Consider, for example, the sequence
\begin{equation}\label{E:counterexample-xi}
\beta_n=
\begin{cases}
2^{2^m}, & \text{if } n=2^m \text{ for some } m\geq 0,\\
1, & \text{otherwise.}
\end{cases}
\end{equation}
Clearly $\beta\notin\ell^2$, thus $\xi_k$ is well-defined for all $k$. For $k\geq 2$, we prove that  $\xi_k = 2^m$ where $m$ is the unique integer satisfying $2^m < k \leq 2^{m+1}$. In particular, we show that if $2^m < k \leq 2^{m+1}$, then \(  B(2^m-1)\leq 2^k < B(2^m). \) 

 Counting the $m+1$ powers $2^0,\ldots,2^m$ among the integers $0,\ldots,2^m$, we obtain that
\begin{equation}\label{E:B2m}
B(2^m) = \sum_{l=0}^{m} 4^{2^l} + (2^m - m),
\qquad m\geq 0.
\end{equation}
In particular, 
\[
B(2^m)> 4^{2^m}=4^{\frac{1}{2}\cdot 2^{(m+1)} }=2^{2^{m+1}}\geq  2^k, \qquad m\geq 1.
\]
From \eqref{E:B2m}, we also deduce that
\begin{equation}\label{E:B2m-1}
B(2^m-1) = B(2^m)-4^{2^m}
= \sum_{l=0}^{m-1} 4^{2^l} + (2^m - m),
\qquad m\geq 1.
\end{equation}
We claim that
\begin{equation}\label{E:Bdyadic}
 B(2^m-1) < 2\cdot 4^{2^{m-1}} = 2^{2^m+1},
\qquad m\geq 2 .
\end{equation}
We rewrite $B(2^m-1)$ as
\[
B(2^m-1) = 4^{2^{m-1}} + S_m
+ (2^m - m),
\]
where $S_m:=\sum_{l=0}^{m-2} 4^{2^l}$. It suffices to prove that $S_m + (2^m-m) < 4^{2^{m-1}}$ for $m\geq 2$. We estimate $S_m$ as follows:
\[
S_m =\sum_{l=0}^{m-2} 4^{2^l}\leq \sum_{i=1}^{2^{m-2}} 4^{i}
= \frac{4^{\,2^{m-2}+1}-4}{3}
< \frac{4}{3}\,4^{2^{m-2}}
= \frac{4}{3}\,2^{2^{m-1}}, \qquad m    \geq 2 .
\]
Since $2^{m-1}\leq 2^{m}-1$ for $m\geq 1$, we have that
$2^{2^{m-1}}\leq 2^{\,2^m-1}=\tfrac12\,4^{2^{m-1}}$, and therefore
\[
S_m \;<\; \frac{4}{3}\cdot\frac12\,4^{2^{m-1}} \;=\; \frac{2}{3}\,4^{2^{m-1}} .
\]
For the remaining summand, notice that
\[
3(2^m-m) \leq 3\cdot 2^m \leq 2^{2^m}=4^{2^{m-1}}, \qquad m\geq 2.
\]
We conclude that
\[
S_m + (2^m-m) \;<\; \frac{2}{3}\,4^{2^{m-1}} + \frac{1}{3}\,4^{2^{m-1}}
\;=\; 4^{2^{m-1}},
\]
which establishes \eqref{E:Bdyadic}, as claimed. Therefore, we conclude that $\xi_k = 2^m$ for all
\(k \in [2^m+1, 2^{m+1}]\), where \(m\geq 2.\) Finally, we show that the quotient $B(\xi_k)/2^k$ is unbounded. Choosing $k=2^m+1$,
\[
\frac{B(\xi_k)}{2^k}
= \frac{B(2^m)}{2^{2^m+1}}
\geq  \frac{4^{2^m}}{2\cdot 2^{2^{m}}}
=\frac{1}{2}\cdot 2^{2^m}\xrightarrow[m\to\infty]{}+\infty .
\]
Therefore, no finite constant $C$ satisfies $B(\xi_k)\leq C\,2^k$ for all
$k$.

For what follows, we restrict ourselves to a special class: we only consider the sequences $\beta$ {such that $\beta_j\neq0$ for all $j$ and}
\begin{equation}\label{extra assumption}
    C_{\beta}:= \sup_j\frac{|\beta_{j+1}|}{|\beta_{j}|} <\infty.
\end{equation}
{In the rest of the paper, when invoking this property we will ask that $C_\beta<\infty$. Even if we do not specify, we also implicitly ask that $\beta_j\neq0$ for all $j$.}  With this assumption, 
\begin{equation} \label{E:boundBextraassumption}
    2^k\leq B(\xi_k) =B(\xi_k-1)+|\beta_{\xi_k}|^2\leq  (1+C_{\beta}^2)B(\xi_{k}-1)\leq (1+C_{\beta}^2) 2^k.
\end{equation}

{For the sequence $\beta$ introduced in \eqref{E:counterexample-xi},
\[
C_\beta \geq \limsup_m \frac{\beta_{2^m}}{\beta_{2^m-1}}=\lim_m 2^{2^m}=+\infty.
\]
Notice that $\beta_n\leq 2^n$ for every $n$, and that the geometric sequence $b_n=2^n$ satisfies
\[
C_b=\sup_j \frac{2^{j+1}}{2^j}=2.
\]
This shows that the fact that the quotient $B(\xi_k)/2^k$ is unbounded for the sequence $\beta$ is not due to the growth of the sequence, rather to the oscillation. In the continuous setting, this problem does not occur since $\xi_k$ is allowed to be a real number.}

The next lemma is the equivalent of Lemma 5.1 of \cite{BDSjlms2025}. It is the reason why we need the extra assumption \eqref{extra assumption} on $\beta$. We introduce the notation $Z_k:=[\xi_k,\xi_{k+1}-1]$.

\begin{lemma}\label{lemma 17}
   Given the interval $I=[a,b]$, let  
   \begin{align*}
  A(I)&:=\sup_{m\in I}\bigg(\sum_{j=m+1}^b|\beta_j|^2\bigg)^{\frac{1}{2}}\biggl(\sum_{k=a}^{m}|\alpha_k|^{2}\biggr)^{\frac{1}{2}}, \\
  B(I)&:=\sup_{m\in I} \bigg(\sum_{j=a}^{m-1}|\beta_j|^2\bigg)^{\frac{1}{2}}\biggl(\sum_{k=m}^{b}|\alpha_k|^{2}\biggr)^{\frac{1}{2}}.
\end{align*}
Assume that $C_{\beta}:= \sup_j\frac{|\beta_{j+1}|}{|\beta_{j}|} <\infty$. Then, for all $k\geq 0$,
\[
A(Z_{k}\cup Z_{k+1})\geq C_B\sigma_{k}, \qquad B(Z_{k}\cup Z_{k+1})\geq \frac{C_B}{2}\sigma_{k+1},
\]
where $C_B=(1+C_\beta^2)^{-\frac{1}{2}}$.
\end{lemma}
\begin{proof}
First, notice that by definition of $\xi_k$, we have that
    \begin{align*}
        2^{k+1}&<B(\xi_{k+1}) = B(\xi_k -1)+ \sum_{j=\xi_k}^{\xi_{k+1}-1}|\beta_j|^2+|\beta_{\xi_{k+1}}|^2\\
        &\leq 2^k +  \sum_{j=\xi_k}^{\xi_{k+1}-1}|\beta_j|^2 + C_\beta^2|\beta_{\xi_{k+1}-1}|^2\leq 2^k +(1+C_\beta^2)  \sum_{j=\xi_k}^{\xi_{k+1}-1}|\beta_j|^2.
    \end{align*}
    In particular, we deduce that for every $k\geq 0$
    \[
    \sum_{j=\xi_k}^{\xi_{k+1}-1}|\beta_j|^2\geq (1+C_\beta^2)^{-1}(2^{k+1}-2^k) = C_B^2 2^k.
    \]
 On the one hand, by \eqref{E:sigmabounds},
\begin{align*}
 A(Z_{k}\cup Z_{k+1})&= \sup_{\xi_k\leq m\leq \xi_{k+2}-1}\bigg(\sum_{j=m+1}^{\xi_{k+2}-1}|\beta_j|^2\bigg)^{\frac{1}{2}}\biggl(\sum_{j=\xi_k}^{m} |\alpha_j|^2\biggr)^{\frac{1}{2}}\\
 &\geq \bigg(\sum_{j=\xi_{k+1}}^{\xi_{k+2}-1}|\beta_j|^2\bigg)^{\frac{1}{2}}\biggl(\sum_{j=\xi_k}^{\xi_{k+1}-1} |\alpha_j|^2\biggr)^{\frac{1}{2}}\\
 &\geq C_B \biggl(2^{k+1}\sum_{j=\xi_k}^{\xi_{k+1}-1} |\alpha_j|^2\biggr)^{\frac{1}{2}}\geq C_B\sigma_k.
\end{align*}
Similarly, for $B$ we have that
\begin{align*}
 B(Z_{k}\cup Z_{k+1})&= \sup_{\xi_k\leq m\leq \xi_{k+2}-1}\bigg(\sum_{j=\xi_k}^{m-1}|\beta_j|^2\bigg)^{\frac{1}{2}}\biggl(\sum_{k=m}^{\xi_{k+2}-1}|\alpha_k|^{2}\biggr)^{\frac{1}{2}}\\
 &\geq \bigg(\sum_{j=\xi_{k}}^{\xi_{k+1}-1}|\beta_j|^2\bigg)^{\frac{1}{2}}\biggl(\sum_{j=\xi_{k+1}}^{\xi_{k+2}-1} |\alpha_j|^2\biggr)^{\frac{1}{2}}\geq \frac{C_B}{2}\sigma_{k+1}.
\end{align*}
\end{proof}
Using this lemma, we are able to prove the converse to Theorem \ref{T:sigmaimpliesSchatten}.

\begin{theorem}\label{T:Schattenimpliessigma}
Let $\alpha\in\ell^2$, and $\beta\notin\ell^2$ be such that $C_{\beta}:= \sup_j\frac{|\beta_{j+1}|}{|\beta_{j}|} <\infty$. For $p>0$, if the sequence of approximation numbers $(a_{k}(F_{\alpha,\beta}))_k$ is in $\ell^p$, then $\sigma$ belongs to $\ell^p$.
\end{theorem}
\begin{proof} 
The proof follows the same steps of Theorem 5.5 in \cite{BDSjlms2025}. It relies on the following lemmas, that hold as a consequence of Lemma \ref{lemma 17}. We recall that $C_B=(1+C_\beta^2)^{-\frac{1}{2}}$.

\begin{itemize}
    \item (\cite[Lemma 5.2]{BDSjlms2025})
\emph{For $\epsilon > 0$ and an interval $I$, consider the set  
\begin{equation*}
S(\epsilon):=\{k\in\mathbb{N}: Z_k\subseteq I, \sigma_{k}>\epsilon\}.
\end{equation*}
If $\#S(\epsilon)\geq 4$, then $J( I ) > \dfrac{C_B}{2}\epsilon$ and $L(I)>\dfrac{C_B}{2}\epsilon$. }

    \item (\cite[Lemma 5.3]{BDSjlms2025}) \emph{
    Let $t>0$, $\epsilon_t=\dfrac{C_B}{2}t$ and let $N = N(\epsilon_t)$ be the length of the $(\epsilon_t,L)$-sequence $\{c_i\}_{i=0}^N$. Then,
\[\#\{k\in\mathbb{N}: \sigma_{k}>t\}\leq 5N(\epsilon_t)+3.\]
}

    \item (\cite[Lemma 5.4]{BDSjlms2025}) \emph{
 Let $t>0$ satisfy $\mathcal{L}<\dfrac{C_B t}{2}\leq L(0,\infty)$.  Then,
\[\#\{k\in\mathbb{N}:\sigma_{k}>t\}\leq 5\#\left\lbrace k\in\mathbb{N}:a_k(F_{\alpha,\beta})> \frac{C_B t}{4}\right\rbrace+3.\]
}
\end{itemize}
As a consequence of these three lemmas, one can show that
 \begin{align*}
\|\sigma\|^{p}_{p}&\leq C_p \|(a_{k}(F_{\alpha,\beta}))_k\|_{p}^{p}+3\sup_k|\sigma_k|^{p} <+\infty.
 \end{align*} 
\end{proof}

\subsection{Proof of Theorem \ref{T:IntroSchattenrho}}
In the previous subsection, we proved the following.
\begin{theorem}\label{T:SchattenFactorablegeneral}
    Given sequences $\alpha\in\ell^2,\beta\notin\ell^2$ and $1<p<\infty$, if $\sigma\in\ell^p$ then the factorable matrix $F_{\alpha,\beta}$ is in the $p$-Schatten class $\mathcal{S}^p(\ell^2)$. Moreover, if $\sup_j\frac{|\beta_{j+1}|}{|\beta_{j}|} <\infty$, then the converse is also true.
\end{theorem} We connect this characterization with the sequence $\rho$ defined in \eqref{def of rho}. 

\begin{lemma}\label{L:etaxikesigma}
    Consider $\alpha\in\ell^2$, $\beta\notin\ell^2$ and $1<p<\infty$. If $(\rho_{\xi_k})_{k\in\mathbb{N}}$ is in $\ell^q$, then $\sigma\in\ell^q$. Moreover, if $C_\beta=\sup_j \frac{|\beta_{j+1}|}{|\beta_{j}|}<\infty$, the converse is also true.
\end{lemma}
\begin{proof}
 On the one hand, we have that
\begin{align*}
 \sigma_k&= \bigg(\sum_{j=\xi_k}^{\xi_{k+1}-1}B(j)|\alpha_j|^2\bigg)^{\frac{1}{2}}\leq B(\xi_{k+1}-1)^{\frac{1}{2}}  \bigg(\sum_{j=\xi_k}^{\infty}|\alpha_j|^2\bigg)^{\frac{1}{2}}\\
 &\leq (2^{k+1})^{\frac{1}{2}}  \bigg(\sum_{j=\xi_k}^{\infty}|\alpha_j|^2\bigg)^{\frac{1}{2}}\leq \sqrt{2} B(\xi_k)^{\frac{1}{2}}\bigg(\sum_{j=\xi_k}^{\infty}|\alpha_j|^2\bigg)^{\frac{1}{2}}=\sqrt{2}  \rho_{\xi_k}.
\end{align*}
Conversely, if $C_\beta<\infty$, notice that for $k\in\mathbb{N}$
\[
2^k\leq B(\xi_{k})=B(\xi_{k}-1)+|\beta_{\xi_k}|^2\leq (1+C_\beta^2)B(\xi_{k}-1)\leq (1+C_\beta^2)2^k.
\]
It follows that
\begin{align*}
 \sum_{k=0}^\infty \rho_{\xi_k}^q&= \sum_{k=0}^\infty B(\xi_{k})^{\frac{q}{2}}  \left(\sum_{j=\xi_k}^\infty |\alpha_j|^2\right)^\frac{q}{2}\\
 &\leq \sum_{k=0}^\infty   \bigg( B(\xi_k)\sum_{h=k}^\infty \frac{1}{2^h}\sum_{j=\xi_h}^{\xi_{h+1}-1}B(j)|\alpha_j|^2\bigg)^\frac{q}{2}\\
 &\leq (1+C_\beta^2)^{\frac{q}{2}} \sum_{k=0}^\infty  \bigg(\sum_{h=k}^\infty 2^{k-h} \sigma_h^2\bigg)^\frac{q}{2}.
\end{align*}
If $q\leq 2$, then by sub-additivity of the power $q/2\leq 1$
\begin{align*}
   \sum_{k=0}^\infty  \bigg(\sum_{h=k}^\infty 2^{k-h}\sigma_h^2\bigg)^\frac{q}{2} &\leq \sum_{k=0}^\infty  \sum_{h=k}^\infty 2^{\frac{q(k-h)}{2}}\sigma_h^q\\
 &\leq \sum_{h=0}^\infty \sigma_h^q  2^{-\frac{qh}{2}}\sum_{k\leq h} 2^{\frac{qk}{2}}\\
 &\leq \sum_{h}\sigma_h^q  2^{\frac{-qh}{2}} \frac{1-2^{\frac{q(h+1)}{2}}}{1-2^{\frac{q}{2}}}\leq C_q \sum_{h} \sigma_{h}^q.
\end{align*}
If $q\geq 2$, we rewrite
\[
 \sum_{k=0}^\infty  \bigg(\sum_{h=k}^\infty 2^{k-h} \sigma_h^2\bigg)^\frac{q}{2}=  \sum_{k=0}^\infty \bigg(\sum_{m=0}^\infty 2^{-m}\sigma_{k+m}^2\bigg)^\frac{q}{2}.
\]
If $\sigma\in\ell^q$, then for every $m\in\mathbb{N}$ the sequence $x^{(m)}:=(2^{-m}\sigma_{k+m}^2)_{k\in\mathbb{N}}$ is in the Banach space $\ell^{\frac{q}{2}}$, and
\begin{align*}
    \sum_{m=0}^\infty \|x^{(m)}\|_{\frac{q}{2}}&= \sum_{m=0}^\infty \bigg(\sum_{k=0}^\infty 2^{-\frac{mq}{2}}\sigma_{k+m}^q\bigg)^{\frac{2}{q}}\\
    &\leq \sum_{m=0}^\infty 2^{-m} \bigg(\sum_{k=0}^\infty \sigma_{k}^q\bigg)^{\frac{2}{q}}=2\|\sigma\|_q^2.
\end{align*}
Then, the series $\sum_{m\in\mathbb{N}} x^{(m)}$ converges in $\ell^{\frac{q}{2}}$ and
\begin{align*}
\Big\|\sum_{m\in\mathbb{N}} x^{(m)}\Big\|_{\frac{q}{2}}^\frac{q}{2} =\sum_{k=0}^\infty \bigg(\sum_{m=0}^\infty 2^{-m}\sigma_{k+m}^2\bigg)^\frac{q}{2} \leq \Big(\sum_{m\in\mathbb{N}} \|x^{(m)}\|_{\frac{q}{2}}\Big)^\frac{q}{2}  \leq 2^\frac{q}{2} \|\sigma\|_q^q.
 \end{align*}

\end{proof}

In the following theorem, we summarise the relations between the sequence $\rho$ and the Schatten class of $F_{\alpha,\beta}$.

\begin{theorem}\label{T:thmSchattenrho}
    Let $\alpha\in\ell^2,\beta\notin\ell^2$ and $1<p<\infty$. If the subsequence $(\rho_{\xi_k})_{k\in\mathbb{N}}$ is in $\ell^p$, then $F_{\alpha,\beta}\in \mathcal{S}^p(\ell^2)$. In particular, if $\rho\in\ell^p,$ then $F_{\alpha,\beta}\in \mathcal{S}^p(\ell^2)$. Moreover, assuming that $C_\beta=\sup_j \frac{|\beta_{j+1}|}{|\beta_{j}|}<\infty$, if $F_{\alpha,\beta}\in \mathcal{S}^p(\ell^2)$ then $(\rho_{\xi_k})_{k\in\mathbb{N}}$ is in $\ell^p$. However, in this case, $\rho$ is not necessarily in $\ell^p$.
\end{theorem}
\begin{proof}
     The relation between the Schatten class properties of $F_{\alpha,\beta}$ and the sequence $(\rho_{\xi_k})_{k\in\mathbb{N}}$ is clear from Lemma \ref{L:etaxikesigma}. To conclude the proof, we show that there exists a sequence $\rho$ not in $\ell^p$ such that the subsequence $(\rho_{\xi_k})_{k\in\mathbb{N}}$ is in $\ell^p$. The example is actually a Rhaly matrix. The idea is simple: take $\beta\equiv 1$ and $\alpha$ such that
    \[
    \sum_{k=m}^\infty |\alpha_k|^2 \asymp \frac{1}{(m+1)^{1+\frac{1}{p}}}, \qquad m\to\infty.
    \]
    An explicit example can be constructed using the telescopic identity
    \begin{equation*} \sum_{k=m}^{\infty} \left( \frac{1}{(k+1)^{1+\frac{1}{p}}} - \frac{1}{(k+2)^{1+\frac{1}{p}}} \right) = \frac{1}{(m+1)^{1+\frac{1}{p}}}.
\end{equation*}
Choosing then
\[
\alpha_k = \left(\frac{1}{(k+1)^{1+\frac{1}{p}}} - \frac{1}{(k+2)^{1+\frac{1}{p}}}\right)^{\frac{1}{2}},
\]
it is easy to check that $\alpha\in\ell^2$. Since in this case $\xi_k=2^k,$
 \[
\rho_{2^k}=B(2^k)\sum_{j=2^k}^\infty|\alpha_k|^2=(2^k+1)\frac{1}{(2^k+1)^{1+\frac{1}{p}}}
 \]
 is in $\ell^p$, whereas
 \[
 \rho_m = B(m)\sum_{j=m}^\infty|\alpha_k|^2=\frac{1}{(m+1)^{\frac{1}{p}}}
 \]
 does not belong to $\ell^p$.
\end{proof}

Now, Theorem \ref{T:IntroSchattenrho} is actually a corollary of the previous result.

\subsection{Proof of Theorem \ref{T:betapolynomial}} For this part, take $\alpha\in\ell^2$ and $1<p<\infty$. By assumption, $\beta$ satisfies
    \[
    C_1 (n+1)^\delta \leq |\beta_n| \leq C_2(n+1)^\delta, \qquad n\geq 0,
    \]
    for $\delta>-\frac{1}{2}$ and $C_1,C_2>0$. Notice that $\beta$ does not belong to $\ell^2$, for 
    \[
    |\beta_n|^2 \geq C_1^2(n+1)^{2\delta}\geq \frac{C_1^2}{n+1} , \qquad n\geq 0,
    \]
    and that
    \[
   \sup_n \frac{|\beta_{n+1}|}{|\beta_n|}\leq \sup_n \frac{C_2}{C_1}\left(1+\frac{1}{n+1}\right)^\delta  <\infty.   \]
Then, by Theorems \ref{T:sigmaimpliesSchatten} and \ref{T:Schattenimpliessigma}, $F_{\alpha,\beta}$ belongs to $\mathcal{S}^p(\ell^2)$ if and only if the sequence $\sigma$ is in $\ell^p.$ By \eqref{E:sigmabounds},  
\begin{equation*}
        2^k \sum_{j=\xi_k}^{\xi_{k+1}-1}|\alpha_j|^2 \leq \sigma_k^2 \leq  2^{k+1} \sum_{j=\xi_k}^{\xi_{k+1}-1}|\alpha_j|^2.
    \end{equation*}

We recall that the sequence $\xi_k$ is defined as
\[
 \xi_k := \inf\{n \in \mathbb{N} \colon B(n) > 2^k\},
\]
where
\[
B(n)=\sum_{k=0}^n|\beta_n|^2 \asymp \sum_{k=0}^n (k+1)^{2\delta},
\]
uniformly in $n\geq 0$. With a standard trick, estimating the sum with the integral, one has that for $n$ big enough
\[
B(n) \asymp (n+1)^{2\delta+1}. 
\]
The fact that $\delta>-\frac{1}{2}$ is crucial. Knowing also that $ B(\xi_n)\asymp 2^n$ (see \eqref{E:boundBextraassumption}), we conclude that \(\xi_k \asymp a^k\), where $a=2^{\frac{1}{2\delta+1}}$. Thus, it follows that 
\begin{equation} \label{E:asigma}
    K_1\sum_{k=0}^\infty\left(2^k\sum_{j=[a^k]}^{[a^{k+1}]-1}\!\!\!|\alpha_j|^2\right)^{\frac{p}{2}}\!\!\!\leq \sum_{k=0}^\infty \left(2^k \sum_{j=\xi_k}^{\xi_{k+1}-1}\!\!\!|\alpha_j|^2\right)^{\frac{p}{2}}\!\!\! \leq K_2\sum_{k=0}^\infty\left(2^k\sum_{j=[a^k]}^{[a^{k+1}]-1} \!\!\!|\alpha_j|^2\right)^{\frac{p}{2}},
\end{equation}

for constants $K_1,K_2>0$. These constants are hiding the following fact: an interval $[\xi_k,\xi_{k+1}-1]$ might contain or be contained in multiple adjacent intervals $[a^k,a^{k+1}-1]$, but the maximum number of such intervals is uniformly bounded since $\xi_k\asymp a^k$.

At this point, we want to rewrite this quantity changing the $a$-adic interval in the more standard $2$-adic interval. We will need the following technical results. This part is somewhat folklore for the experts in Besov spaces (see for example \cite{T88}), and it was kindly provided to us by Mattia Calzi, whom we thank. We include the details for the sake of completeness.

\begin{lemma}\label{lem:1}
		Take  $a_1,a_2>1$  and $\lambda>0$. Then, there is a constant $C>0$ such that
		\[
		\Bigg{|}\!\Bigg{|}  \sum_{j\in\mathbb{N}} \min\left(\frac{a_1^{\lambda j}}{a_2^{\lambda k}},\frac{a_2^{\lambda k}}{a_1^{\lambda j}} \right)|b_j| \Bigg{|}\!\Bigg{|} _{\ell^p_k}\leq C\|{b}\|_{p}
		\]
		for every sequence $b$.
	\end{lemma}
	We recall that with the symbol $\ell^p_k$ we denote that the sequence is indexed on $k$.
	\begin{proof}
    Set  \[
        c_{j,k}:= \min\left(\frac{a_1^{j}}{a_2^{ k}},\frac{a_2^{ k}}{a_1^{ j}} \right), \qquad j,k\in\mathbb{N}.
        \]
		If $p\geq 1$, the statement follows from Schur's lemma, see \cite[Appendix A]{grafakos2009modern}, for
        \[
		\sup_{j\in \mathbb{N}} \sum_{k\in\mathbb{N}} c_{j,k}^\lambda + \sup_{k\in \mathbb{N}} \sum_{j\in\mathbb{N}} c_{j,k}^\lambda <\infty,
		\]
        for every $\lambda >0$. If, otherwise, $p<1$, then observe that
		\[
		\Bigg{|}\!\Bigg{|}  \sum_{j\in\mathbb{Z}} c_{j,k}^\lambda |b_j| \Bigg{|}\!\Bigg{|} _{\ell^p_k(\mathbb{Z})}^p\leq \Bigg{|}\!\Bigg{|}  \sum_{j\in\mathbb{Z}} c_{j,k}^{\lambda p}|b_j|^p \Bigg{|}\!\Bigg{|} _{\ell^1_k(\mathbb{Z})}\leq \left(\sup_{j\in \mathbb{N}} \sum_{k\in\mathbb{N}} c_{j,k}^{\lambda p} \right)\||b_j|^p\|_{\ell^1_j(\mathbb{Z})},
		\]
		whence the conclusion. 
	\end{proof}

	\begin{lemma}\label{lem:2}
		Take $a_1,a_2>1$ and $\lambda>0$. For every sequence $b$ and for every $j\geq 0$, define $\Delta^{(1)}_j(b)= \sqrt{\sum_{ a_1^j\leq |k|<a_1^{j+1}  }|b_k|^2 }$ and $\Delta^{(2)}_j(b)= \sqrt{\sum_{ a_2^j\leq |k|<a_2^{j+1}  }|b_k|^2 }$. Then, there is a constant $C>0$ such that
		\[
		\|  a_1^{\lambda j} \Delta^{(1)}_j b   \|_{\ell^p_j(\mathbb{N})}\leq C {\|a_2^{\lambda j} \Delta^{(2)}_j(b)\|}_{\ell^p_j(\mathbb{N})}.
		\]
	\end{lemma}
	\begin{proof}
		For every $j\in\mathbb{N}$, take $k_j\geq 0$ so that 
        \begin{equation}\label{eq utile1}
                    a_2^{k_j}\leq a_1^j< a_2^{k_j+1}.
        \end{equation}
        Observe that
		\[
		(\Delta^{(1)}_j b)^2\leq \sum_{k= k_j }^{k_{j+1}+1} (\Delta^{(2)}_k b)^2\leq \sum_{k\geq  k_j } (\Delta^{(2)}_k b)^2
		\]
		for every $j\geq 0$.
		Consequently,
		\[
		a_1^{2\lambda j }(\Delta^{(1)}_j b)^2\leq \sum_{k\geq  k_j }  \frac{a_1^{2\lambda j}}{a_2^{2\lambda k}}  a_2^{2\lambda k}(\Delta^{(2)}_k b)^2\leq a_2^{4\lambda}\sum_{k\geq 0 }  \min\left(\frac{a_1^{2\lambda j}}{a_2^{2\lambda k}},\frac{a_2^{2\lambda k}}{a_1^{2\lambda j}} \right)  a_2^{2\lambda k}(\Delta^{(2)}_k b)^2,
		\]
		since $a_1^{2\lambda j} /a_2^{2\lambda k}\leq a_2^{4\lambda} a_2^{2\lambda k}/a_1^{2\lambda j} $ for every $k\geq k_j$ due to \eqref{eq utile1}.
		 Lemma \ref{lem:1} then shows that there is a constant $C>0$ such that
		\begin{align*}
			\|{  a_1^{\lambda j} \Delta^{(1)}_j b   }\|_{\ell_j^p}&\leq a_2^{2\lambda}\Bigg{|}\!\Bigg{|} {  \sum_{k\geq 0 }  \min\left(\frac{a_1^{2\lambda j}}{a_2^{2\lambda k}},\frac{a_2^{2\lambda k}}{a_1^{2\lambda j}} \right)   a_2^{2\lambda k}(\Delta^{(2)}_k b)^2 }\Bigg{|}\!\Bigg{|} _{\ell_j^{\frac{p}{2}}}^\frac{1}{2}\\
				&\leq a_2^{ 2\lambda } C^\frac{1}{2} \|{ a_2^{2\lambda k}(\Delta^{(2)}_k b)^2 }\|_{\ell^{\frac{p}{2}}_k}^\frac{1}{2}=a_2^{ 2\lambda } C^\frac{1}{2} \|{ a_2^{\lambda k}\Delta^{(2)}_k b }\|_{\ell^{p}_k},
		\end{align*}
		whence the result.
\end{proof}

Taking $a_1=a=2^{\frac{1}{2\delta+1}}$, $a_2=2$, and $\lambda = \frac{1}{2}\log_a 2=\delta+\frac{1}{2}>0$, we have that
\begin{align*}
    \left(2^k\sum_{j=[a^k]}^{[a^{k+1}]-1}|\alpha_j|^2\right)^{\frac{1}{2}}=a_1^{\lambda k} \Delta^{(1)}_k (|\alpha|^2), \qquad k\in\mathbb{N},
\end{align*}
and 
\begin{align*}
   a_2^{\lambda k} \Delta^{(2)}_k (|\alpha|^2)= \left(2^{(2\delta+1)k}\sum_{j=2^k}^{2^{k+1}-1}|\alpha_j|^2\right)^{\frac{1}{2}}, \qquad k\in\mathbb{N}.
\end{align*}
By \eqref{E:asigma} and Lemma \ref{lem:2} (applied twice, inverting the roles of $a_1,a_2$ the second time), we conclude that $\sigma\in\ell^p$ if and only if
\[
\sum_{k=0}^\infty \left(2^{(2\delta+1)k}\sum_{j=2^k}^{2^{k+1}-1} |\alpha_j|^2\right)^{\frac{p}{2}}<\infty,
\]
and the proof is complete.

\section{Weighted Hardy spaces}\label{S:weightedHardy}
We recall that given a sequence of positive real numbers $\omega=(\omega_n)_n$, the associated weighted Hardy space is defined as 
\begin{equation*}
    H^2(\omega) := \left\{f(z) = \sum_{n=0}^\infty a_n z^n \in \operatorname{Hol}(\mathbb{D}) : \|f\|^2_\omega = \sum_{n=0}^\infty |a_n|^2 \omega_n < \infty \right\}.
\end{equation*}
We review some basic facts. The space $H^2(\omega)$ is a reproducing kernel Hilbert space. The inner product is
\[
\langle \sum_n a_n z^n, \sum_n b_n z^n \rangle_\omega := \sum_n a_n\overline{b_n}\omega_n.
\] 
The reproducing kernel is given by
\begin{equation*}
    k^\omega(z,\lambda) := \sum_{n=0}^\infty \frac{\overline{\lambda}^n}{\omega_n}z^n, \qquad z,\lambda\in\mathbb{D}.
\end{equation*}
The set of monomials $\{\varphi_n\}_{n\in\mathbb{N}}$ defined by
\begin{equation}\label{E:orthonormalmonomials}
    \varphi_n(z) := \frac{z^n}{\sqrt{\omega_n}}, \qquad z\in\mathbb{D},
\end{equation}
is a complete orthonormal basis for $H^2(\omega)$. The properties of the backward shift operator 
\[
Lf(z) := \frac{f(z)-f(0)}{z}, \qquad z\in\mathbb{D},
\]
acting on $H^2(\omega)$, can be entirely characterized in terms of the weight sequence $\omega$ \cite{CM1995}. A straightforward computation yields $L\varphi_0 = 0$ and 
\[
L\varphi_{n} = \sqrt{\frac{\omega_{n-1}}{\omega_n}}\varphi_{n-1}, \quad  n>0.
\]
As a consequence, $L$ is a bounded operator on $H^2(\omega)$ if and only if $\sup_{j\geq 0} (\omega_j/\omega_{j+1}) < \infty$ \cite[Proposition 2.7]{CM1995}. We discuss some specific examples. Given a parameter $\gamma > -1$, the weighted Bergman space $A^2_\gamma$ is defined as

\begin{equation*}
A^2_\gamma = \{f\in\operatorname{Hol}(\mathbb{D})\colon  \int_{\mathbb{D}} |f(z)|^2 \, (1 - |z|^2)^\gamma \operatorname{d}\!A(z) < \infty\},
\end{equation*}
where $A$ is the bidimensional Lebesgue measure. It is well-known that $A^2_\gamma=H^2(\omega_\gamma)$, with $\omega_{\gamma,n}=(n+1)^{-(\gamma+1)}, n\geq 0$. The case $\gamma=0$ gives the standard Bergman space on the unit disk. Given a parameter $\alpha\in[0,1]$, the weighted Dirichlet space $\mathcal{D}_\alpha$ is
\begin{equation*}
\mathcal{D}_\alpha = \{f\in\operatorname{Hol}(\mathbb{D})\colon  \int_{\mathbb{D}} |f'(z)|^2 \, (1 - |z|^2)^\alpha \operatorname{d}\!A(z) < \infty\}.
\end{equation*}
It is well-known that $\mathcal{D}_\alpha=H^2(\operatorname{w}_\alpha)$, with $\operatorname{w}_{\alpha,n}=n^{1-\alpha}, n\geq 1$. They provide a natural scale of spaces linking the classical Dirichlet space $\mathcal{D}:=\mathcal{D}_0$ to the Hardy space $\mathcal{D}_1=H^2$. We could consider $\mathcal{D}_\alpha$ also for $\alpha>1$, but we would obtain the weighted Bergman space $A^2_{\alpha-2}$. For some references and results on weighted Hardy spaces see for example \cite{CP,GP,LLQR,Z}.

Given a sequence of complex numbers $\eta=(\eta_n)_n$, we define the Rhaly operator $R_\eta$, which acts on sequences $x=(x_n)_n$ as 
\begin{equation*}
    (R_\eta x)_n = \eta_n \sum_{j=0}^n x_j, \qquad n\in\mathbb{N}.
\end{equation*}
By identifying an analytic function $f(z) = \sum_n a_n z^n$ with its sequence of Taylor coefficients, we can consider the action of $R_\eta$ on $\operatorname{Hol}(\mathbb{D})$ via the formula
\[
R_\eta f(z) = \sum_{n=0}^\infty \left(\eta_n \sum_{j=0}^n a_j\right) z^n, \qquad z\in\mathbb{D}.
\]

Given a weight $\omega$, we compute the matrix representation of the Rhaly operator $R_\eta$ acting on $H^2(\omega)$ with respect to the orthonormal basis $\{\varphi_n\}_n$ introduced in \eqref{E:orthonormalmonomials}. For $z\in\mathbb{D}$, we have that
\begin{align*}
    R_\eta \varphi_j(z) &= \sum_{k=0}^\infty \left( \eta_k \sum_{i=0}^k \frac{\delta_{j,i}}{\sqrt{\omega_j}}\right) z^k \\
    &= \frac{1}{\sqrt{\omega_j}}\sum_{k=j}^\infty \eta_k z^k = \frac{1}{\sqrt{\omega_j}}\sum_{k=j}^\infty \eta_k \sqrt{\omega_k} \varphi_k(z).
\end{align*}
Therefore, the $j$-th column of the matrix associated with $R_\eta$ on $H^2(\omega)$ is the vector whose $k$-th component is $0$ for $k < j$, and $\frac{\eta_k \sqrt{\omega_k}}{\sqrt{\omega_j}}$ for $k \geq j$. We obtain the factorable matrix
\begin{equation}\label{E:matrix Rhaly on weighted Hardy}
\begin{pmatrix}
\frac{\eta_0 \sqrt{\omega_0}}{\sqrt{\omega_0}}  & 0 & 0 & \cdots\\
\frac{\eta_1 \sqrt{\omega_1}}{\sqrt{\omega_0}}  &  \frac{\eta_1 \sqrt{\omega_1}}{\sqrt{\omega_1}}& 0 & \cdots\\
\frac{\eta_2 \sqrt{\omega_2}}{\sqrt{\omega_0}} &  \frac{\eta_2 \sqrt{\omega_2}}{\sqrt{\omega_1}}& \frac{\eta_2 \sqrt{\omega_2}}{\sqrt{\omega_2}}  & \cdots \\
\vdots & \vdots & \vdots  & \ddots
\end{pmatrix}
=
\begin{pmatrix}
\alpha_0 \beta_0 & 0 & 0  & \cdots\\
\alpha_1 \beta_0 &  \alpha_1 \beta_1& 0  & \cdots\\
 \alpha_2\beta_0 &  \alpha_2\beta_1& \alpha_2\beta_2  & \cdots \\
\vdots & \vdots & \vdots  & \ddots
\end{pmatrix},
\end{equation}
where we have set $\alpha_k = \eta_k \sqrt{\omega_k}$ and $\beta_j = 1/\sqrt{\omega_j}$. Notice that the case of the classical Hardy space $(\omega \equiv 1)$ corresponds to the case of Rhaly matrices. We remark that it is very pleasing that the study of Rhaly operators on weighted Hardy spaces corresponds to a very natural generalization of the Rhaly matrices. Also, we obtain at once the proof of Theorem \ref{T:boundcompFactorableweighted}.

\begin{proof}[Proof of Theorem \ref{T:boundcompFactorableweighted}]
    It follows from Theorem 2 of \cite{bennett}, applied to the factorable matrix in \eqref{E:matrix Rhaly on weighted Hardy}, that is, with $\alpha_k = \eta_k \sqrt{\omega_k}$ and $\beta_j = 1/\sqrt{\omega_j}$. 
\end{proof}

We apply this result to the examples that we discussed before and we recover some results that are already present in the literature. 
\begin{itemize}
    \item Weighted Bergman spaces: $A^2_\gamma=H^2(\omega_\gamma)$, with $\omega_{\gamma,n}=(n+1)^{-(\gamma+1)},$ for $\gamma>-1$. By Theorem \ref{T:boundcompFactorableweighted}, the operator $R_\eta$ is bounded on $A^2_\gamma$ if and only if
   \begin{equation} \label{E:boundBergman}
 \sum_{k=0}^m |\eta_k|^2(k+1)^{-(\gamma+1)}\left(\sum_{j=0}^k (j+1)^{\gamma+1}\right)^2=O  \left(\sum_{j=0}^m (j+1)^{\gamma+1}\right).
   \end{equation}
   We can rewrite \eqref{E:boundBergman} as
   \[
   \sum_{k=0}^m |\eta_k|^2k^{\gamma+3} =O(m^{\gamma+2}).
   \]
  Using the lemma at page 100 in \cite{duren1970theory}, this is equivalent to $\sum_{k=m}^\infty |\eta_k|^2=O(m^{-1}).$ By Corollary 3.2 of \cite{Bourdon_Shapiro_Sledd_1989}, we conclude that $R_\eta$ is bounded on $A^2_\gamma$ if and only if the analytic function $f_\eta(z)=\sum_{k=0}^\infty \eta_k z^k$ belongs to the mean Lipschitz space $\Lambda(2,\frac{1}{2})$, as found by Galanopoulos--Girela in \cite{GG2025}. Notice that this does not depend on the parameter $\gamma>-1.$
  \item Weighted Dirichlet spaces: $\mathcal{D}_\alpha=H^2(\operatorname{w}_\alpha)$, with $\operatorname{w}_{\alpha,n}=n^{1-\alpha},$ for $\alpha\in [0,1]$.  By Theorem \ref{T:boundcompFactorableweighted}, the operator $R_\eta$ is bounded on $A^2_\gamma$ if and only if
   \begin{equation} \label{E:boundDirichlet}
 \sum_{k=0}^m |\eta_k|^2k^{1-\alpha}\left(\sum_{j=0}^k j^{\alpha-1}\right)^2=O  \left(\sum_{j=0}^m j^{\alpha-1}\right).
   \end{equation}
   We have to distinguish two cases. If $\alpha>0,$ then we can rewrite \eqref{E:boundDirichlet} as
   \[
   \sum_{k=0}^m |\eta_k|^2k^{\alpha+1} =O(m^{\alpha}).
   \]
As before, this is equivalent to $\sum_{k=m}^\infty |\eta_k|^2=O(k^{-1}),$ and we conclude that $R_\eta$ is bounded on $\mathcal{D}_\alpha$ if and only if the analytic function $f_\eta(z)=\sum_{k=0}^\infty \eta_k z^k$ belongs to the mean Lipschitz space $\Lambda(2,\frac{1}{2})$, as found by Galanopoulos--Girela in \cite{GG2025}. However, for $\alpha=0$ the situation changes. We use the other characterization of Theorem \ref{T:boundcompFactorableweighted}. The operator $R_\eta$ is bounded on $\mathcal{D}$ if and only if 
\[
 K_2:=\sup_{m\in\mathbb{N}} \left(\sum_{k=m}^\infty k|\eta_k|^2\right)\left(\sum_{j=0}^m \frac{1}{j}\right)<\infty,
  \]
  which becomes
  \[
  \sum_{k=m}^\infty k|\eta_k|^2 = O\big(\log(m)^{-1}\big).
  \]
  This also appears in \cite[Theorem 2]{Blasco2026Cesarotype}.
\end{itemize}

We move on to the Schatten class properties of {Rhaly operators on weighted Hardy spaces}. As we already discussed in the previous section, we have to ask some extra conditions on the sequence $\beta$ {given by $\beta_j=1/\sqrt{\omega_j}$}.  Interestingly enough, the sequence~$\beta$ satisfies the condition \eqref{extra assumption} if and only if the backward shift operator $L$ is bounded on $H^2(\omega)$. This is a very natural and standard operator-theoretic condition. {The condition $\beta\notin\ell^2$ has also implications on the properties of $H^2(\omega)$. For example, $\sum_j \frac{1}{\omega_j}=+\infty$ if and only if the normalized reproducing kernels $k_{w}$ of $H^{2}(\omega)$ tend to zero weakly as $|w|\to 1$ (see \cite[Theorem 2.17]{CM1995} keeping in mind that in their notation introduced at page 14, $\beta_j^2$ is equal to our $\omega_j$). The condition $\beta\notin\ell^2$ also implies that $H^{2}(\omega)$ is \emph{algebraically consistent}. We refer to Definition 1.5 and Theorem~2.15 of \cite{CM1995} for more details. We will not be using directly these properties, but we believe it is worth mentioning these connections.}

We are in a position to prove Theorem \ref{T:SchattenBergman}.

\begin{proof}[Proof of Theorem \ref{T:SchattenBergman}]
    We begin with the Bergman spaces. Since $A^2_\gamma=H^2(\omega_\gamma)$, with $\omega_{\gamma,n}=(n+1)^{-(\gamma+1)}, n\geq 0$, the Rhaly operator $R_\eta$ on $A^2_\gamma$ is unitarily equivalent to the factorable matrix $F_{\alpha,\beta}$ acting on $\ell^2,$ with  $\alpha_k = \eta_k \sqrt{\omega_{\gamma,k}}$ and $\beta_j = 1/\sqrt{\omega_{\gamma,j}}=(j+1)^{\frac{\gamma+1}{2}}$. Since $\gamma>-1$, setting $\delta:=\frac{\gamma+1}{2}>0$ we can apply Theorem \ref{T:betapolynomial}. Then, $R_\eta$ belongs to the $p$-Schatten class if and only if 
    \[
    \sum_{k=0}^\infty \left(2^{(\gamma+2)k}\sum_{j=2^k}^{2^{k+1}-1} |\eta_j|^2(j+1)^{-(\gamma+1)}\right)^{\frac{p}{2}}<\infty.
    \]
    But since
    \[
    \sum_{j=2^k}^{2^{k+1}-1} |\eta_j|^2(j+1)^{-(\gamma+1)} \asymp 2^{-(\gamma+1)k}\sum_{j=2^k}^{2^{k+1}-1} |\eta_j|^2, \qquad k\geq 0,
    \]
    we conclude that $R_\eta$ belongs to the $p$-Schatten class if and only if 
    \[
    \sum_{k=0}^\infty \left(2^k\sum_{j=2^k}^{2^{k+1}-1} |\eta_j|^2\right)^{\frac{p}{2}}<\infty.
    \]
    We deal with the Dirichlet spaces analogously. Since $\mathcal{D}_\alpha = H^2(\operatorname{w}_\alpha)$, with $\operatorname{w}_{\alpha,n}=n^{1-\alpha}, n\geq 1$, we have that 
    \[
    \beta_j = \frac{1}{\sqrt{\operatorname{w}_{\alpha,j}}} \asymp (j+1)^\delta,
    \]
     where now $\delta:=(\alpha-1)/2 \in (-\frac{1}{2},0]$. We conclude as in the Bergman case. 
\end{proof}

The Dirichlet space $\mathcal{D}_0$ corresponds to the critical case $\delta=-\frac{1}{2}$, and we have to treat it separately. In this case,  \[
    B(n)=\sum_{j=0}^n (j+1) \asymp \log n,
    \]
    for $n$ big enough. Knowing also that $B(\xi_n)\asymp 2^n$, we deduce that $\xi_n$ has a super-exponential growth, and it is more complicated to handle. All the other cases of weighted Bergman and Dirichlet spaces that we treated shared the same characterization for the $p$-Schatten class with the classical Hardy space. We can show that $\mathcal{D}$ does not. For $p=2$, a (compact) infinite matrix $A=(a_{kj})_{k,j}$ belongs to the Schatten class $\mathcal{S}^2(\ell^2)$ if and only if 
    \[
    \sum_{k,j\in\mathbb{N}}|a_{kj}|^2<\infty.
    \]
    See \cite[Theorem 1.22]{zhu2007operator}. For the Rhaly operator $R_\eta$ on $H^2$, since $\alpha=\eta$ and $\beta\equiv 1$, this condition becomes
    \begin{equation} \label{E:final}
        \sum_{k=0}^\infty \sum_{j=0}^k|\eta_k|^2= \sum_{k=0}^\infty (k+1)|\eta_k|^2<\infty.
    \end{equation}
Notice that this is also obtained from Theorem \ref{T:betapolynomial}, with $\delta=0$, since 
\[
2^k\sum_{j=2^k}^{2^{k+1}-1}|\alpha_j|^2\asymp \sum_{j=2^k}^{2^{k+1}-1} j |\alpha_j|^2,
\] 
and the intervals $[2^k,2^{k+1}-1]$ form a partition. However, for the Rhaly operator $R_\eta$ on $\mathcal{D}$, since $\alpha_k=\eta_k\sqrt{k+1}$ and $\beta_j=(j+1)^{\frac{1}{2}}$, 
    \[
    \sum_{k=0}^\infty \sum_{j=0}^k\frac{(k+1)|\eta_k|^2}{j+1}\asymp \sum_{k=0}^\infty (k+1)|\eta_k|^2 \log k<\infty,
    \] 
    which is not equivalent to \eqref{E:final}.

\section{Acknowledgments}
The authors would like to thank Dr. Mattia Calzi for helpful discussions concerning the Besov spaces. The first two authors are members of Gruppo Nazionale per l’Analisi Matematica, la Probabilit\`a e le loro Applicazioni (GNAMPA) of Istituto Nazionale di Alta Matematica (INdAM).


\begin{thebibliography}{10}

\bibitem{BDSjlms2025}
C.~Bellavita, E.~Dellepiane, and G.~Stylogiannis.
\newblock Boundedness, compactness and {S}chatten class for {R}haly matrices.
\newblock {\em Journal of the London Mathematical Society}, 112(4):e70304, 2025.

\bibitem{BMNP}
C.~Bellavita, A.~M. Moreno, G.~Nikolaidis, and J.~A. Pel\'aez.
\newblock Fractional {V}olterra-type operator induced by radial weight acting on {H}ardy space.
\newblock {\em Mathematische Zeitschrift}, 312(2):Paper No. 49, 38, 2026.

\bibitem{bennett}
G.~Bennett.
\newblock Some elementary inequalities.
\newblock {\em The Quarterly Journal of Mathematics}, 38(4):401--425, 12 1987.

\bibitem{BENNETT1991}
G.~Bennett.
\newblock Some elementary inequalities, {III}.
\newblock {\em The Quarterly Journal of Mathematics}, 42(1):149–174, 1991.

\bibitem{Bennett2004}
G.~Bennett.
\newblock Series of positive terms.
\newblock In {\em Orlicz Centenary Volume}, page 29–38. Institute of Mathematics Polish Academy of Sciences, 2004.

\bibitem{Blasco2026Cesarotype}
{O}. Blasco, P.~Galanopoulos, and D.~Girela.
\newblock Ces{\`a}ro-type operators acting on {D}irichlet spaces.
\newblock {\em Revista Matem{\'a}tica Complutense}, 2026.

\bibitem{Bourdon_Shapiro_Sledd_1989}
P.~S. Bourdon, J.~H. Shapiro, and W.~T. Sledd.
\newblock {\em Fourier series, mean Lipschitz spaces and bounded mean oscillation}, page 81–110.
\newblock London Mathematical Society Lecture Note Series. Cambridge University Press, 1989.

\bibitem{CP}
I.~Chalendar and J.~R. Partington.
\newblock Compactness and norm estimates for weighted composition operators on spaces of holomorphic functions.
\newblock In {\em Harmonic analysis, function theory, operator theory, and their applications}, volume~19 of {\em Theta Ser. Adv. Math.}, pages 81--89. Theta, Bucharest, 2017.

\bibitem{CM1995}
C.~C. Cowen and B.~D. MacCluer.
\newblock {\em Composition operators on spaces of analytic functions}.
\newblock Studies in Advanced Mathematics. CRC Press, Boca Raton, FL, 1995.

\bibitem{DS}
P.~Duren and A.~Schuster.
\newblock {\em Bergman spaces}, volume 100 of {\em Mathematical Surveys and Monographs}.
\newblock American Mathematical Society, Providence, RI, 2004.

\bibitem{duren1970theory}
P.~L. Duren.
\newblock {\em Theory of {$H^p$} Spaces}, volume~38 of {\em Pure and Applied Mathematics}.
\newblock Academic Press, New York, 1970.

\bibitem{Edmunds1988}
D.~E. Edmunds, W.~D. Evans, and D.~J. Harris.
\newblock Approximation numbers of certain {V}olterra integral operators.
\newblock {\em Journal of the London Mathematical Society}, s2-37(3):471–489, 1988.

\bibitem{Edmunds1997}
D.~E. Edmunds, W.~D. Evans, and D.~J. Harris.
\newblock Two-sided estimates of the approximation numbers of certain {V}olterra integral operators.
\newblock {\em Studia Mathematica}, 124(1):59--80, 1997.

\bibitem{GG2025}
P.~Galanopoulos and D.~Girela.
\newblock Rhaly operators acting on {H}ardy, {B}ergman, and {D}irichlet spaces.
\newblock {\em Journal of Geometric Analysis}, 36(3):Paper No. 115, 30, 2026.

\bibitem{GP}
E.~A. Gallardo-Guti\'errez and J.~R. Partington.
\newblock Norms of composition operators on weighted {H}ardy spaces.
\newblock {\em Israel Journal of Mathematics}, 196(1):273--283, 2013.

\bibitem{grafakos2009modern}
L.~Grafakos.
\newblock {\em Modern Fourier Analysis}.
\newblock Graduate Texts in Mathematics. Springer New York, 2009.

\bibitem{GrosseErdmann1998}
K.~G. Grosse-Erdmann.
\newblock {\em The blocking technique, weighted mean operators and {H}ardy's inequality}, volume 1679 of {\em Lecture Notes in Mathematics}.
\newblock Springer-Verlag, Berlin, 1998.

\bibitem{HKZ}
H.~Hedenmalm, B.~Korenblum, and K.~Zhu.
\newblock {\em Theory of {B}ergman spaces}, volume 199 of {\em Graduate Texts in Mathematics}.
\newblock Springer-Verlag, New York, 2000.

\bibitem{LLQR}
P.~Lef\`evre, D.~Li, H.{} Queff\'elec, and L.~Rodr\'iguez-Piazza.
\newblock Characterization of weighted {H}ardy spaces on which all composition operators are bounded.
\newblock {\em Analysis \& PDE}, 18(8):1921--1954, 2025.

\bibitem{T88}
H.~Triebel.
\newblock Characterizations of {B}esov-{H}ardy-{S}obolev spaces: a unified approach.
\newblock {\em Journal of Approximation Theory}, 52(2):162--203, 1988.

\bibitem{zhu2007operator}
K.~Zhu.
\newblock {\em Operator Theory in Function Spaces}.
\newblock Mathematical surveys and monographs. American Mathematical Society, 2007.

\bibitem{Z}
N.~Zorboska.
\newblock {\em Composition operators on weighted {H}ardy spaces}.
\newblock ProQuest LLC, Ann Arbor, MI, 1988.
\newblock Thesis (Ph.D.)--University of Toronto (Canada).

\end{thebibliography}
\end{document}